\documentclass[final]{dmtcs-episciences}
\usepackage[utf8]{inputenc}
\usepackage{subfigure}
\usepackage{amsthm,amssymb,amsmath,enumerate,float,tikz}
\usetikzlibrary{decorations.pathreplacing}

\newtheorem{theorem}{Theorem}
\newtheorem{claim}{Claim}
\newtheorem{lemma}[theorem]{Lemma}
\newtheorem{conjecture}[theorem]{Conjecture}

\author{Michael A. Henning\affiliationmark{1}\thanks{Research supported in part by the University of Johannesburg.}
  \and Elena Mohr\affiliationmark{2}
  \and Dieter Rautenbach\affiliationmark{2}}
\title[On the maximum number of minimum total dominating sets in forests]{On the maximum number of minimum total dominating sets in forests}
\affiliation{
Department of Pure and Applied Mathematics, 
University of Johannesburg, 
South Africa\\
Institute of Optimization and Operations Research,
Ulm University, 
Germany}
\received{2018-8-19}
\revised{2018-12-18}
\accepted{2019-1-11}
\begin{document}
\publicationdetails{21}{2019}{3}{3}{4787}
\maketitle
\begin{abstract}
We propose the conjecture that 
every tree with order $n$ at least $2$ 
and total domination number $\gamma_t$ 
has at most 
$\left(\frac{n-\frac{\gamma_t}{2}}{\frac{\gamma_t}{2}}\right)^{\frac{\gamma_t}{2}}$
minimum total dominating sets.
As a relaxation of this conjecture, 
we show that 
every forest $F$ with order $n$, no isolated vertex,
and total domination number $\gamma_t$ 
has at most 
$$\min\left\{\left(8\sqrt{e}\, \right)^{\gamma_t}\left(\frac{n-\frac{\gamma_t}{2}}{\frac{\gamma_t}{2}}\right)^{\frac{\gamma_t}{2}},
(1+\sqrt{2})^{n-\gamma_t},1.4865^n\right\}$$
minimum total dominating sets.
\end{abstract}

\section{Introduction}

A set $D$ of vertices of a graph $G$ is a {\it dominating set} of $G$ if every vertex of $G$ that is not in $D$ has a neighbor in $D$,
and $D$ is a {\it total dominating set} of $G$ if every vertex of $G$ has a neighbor in $D$.
The minimum cardinalities of a dominating set of $G$
and a total dominating set of $G$ 
are the well studied \cite{hahesl,heye}
{\it domination number $\gamma(G)$} of $G$ 
and the
{\it total domination number $\gamma_t(G)$} of $G$,
respectively.
A (total) dominating set is {\it minimal} 
if no proper subset is a (total) dominating set.
A dominating set of $G$ of cardinality $\gamma(G)$
is a {\it minimum dominating set} of $G$,
and 
a total dominating set of $G$ of cardinality $\gamma_t(G)$
is a {\it minimum total dominating set} 
or {\it $\gamma_t$-set} of $G$.
For a graph $G$, let $\sharp\gamma_t(G)$
be the number of minimum total dominating sets of $G$.

Providing a negative answer to a question 
of Fricke et al.~\cite{frhehehu},
Bie\'{n} \cite{bi} showed that 
trees with domination number $\gamma$
can have more than $2^\gamma$ minimum dominating sets.
In fact, Bie\'{n}'s example allows 
to construct forests with  domination number $\gamma$
that have up to $2.0598^{\gamma}$ 
minimum dominating sets.
In \cite{aldamora} Alvarado et al.~showed that 
every forest with domination number $\gamma$
has at most $2.4606^{\gamma}$ 
minimum dominating sets, and 
they conjectured
that every tree with domination number $\gamma$
has $O\left(\frac{\gamma 2^{\gamma}}{\ln \gamma}\right)$
minimum dominating sets.
In the present paper we consider analogous problems
for total domination, which turns out to behave quite differently.
As shown by the star $K_{1,n-1}$ which has 
total domination number $2$ 
but $n-1$ minimum total dominating sets,
the number of minimum total dominating sets
of a tree is not bounded in terms of its total domination number alone, but in terms of both the order and the total domination number. 
In Figure \ref{fig1} we illustrate 
what we believe to be the structure of trees $T$ 
with given order $n$ at least $2$ and total domination number $\gamma_t$
that maximize $\sharp\gamma_t(T)$.

\begin{figure}[H]
\begin{center}
\unitlength 0.7mm 
\linethickness{0.4pt}
\ifx\plotpoint\undefined\newsavebox{\plotpoint}\fi 
\begin{picture}(92,40)(0,0)
\put(1,5){\circle*{1.5}}
\put(21,5){\circle*{1.5}}
\put(41,5){\circle*{1.5}}
\put(81,5){\circle*{1.5}}
\put(11,5){\circle*{1.5}}
\put(31,5){\circle*{1.5}}
\put(51,5){\circle*{1.5}}
\put(91,5){\circle*{1.5}}
\put(6,5){\makebox(0,0)[cc]{$\ldots$}}
\put(6,0){\makebox(0,0)[cc]{$\underbrace{\hspace{1.0cm}}_{{\ell}_1}$}}
\put(26,5){\makebox(0,0)[cc]{$\ldots$}}
\put(26,0){\makebox(0,0)[cc]{$\underbrace{\hspace{1.0cm}}_{{\ell}_2}$}}
\put(46,5){\makebox(0,0)[cc]{$\ldots$}}
\put(46,0){\makebox(0,0)[cc]{$\underbrace{\hspace{1.0cm}}_{{\ell}_3}$}}
\put(86,5){\makebox(0,0)[cc]{$\ldots$}}
\put(86,0){\makebox(0,0)[cc]{$\underbrace{\hspace{1.0cm}}_{{\ell}_k}$}}
\put(6,15){\circle*{1.5}}
\put(26,15){\circle*{1.5}}
\put(46,15){\circle*{1.5}}
\put(86,15){\circle*{1.5}}
\put(6,25){\circle*{1.5}}
\put(26,25){\circle*{1.5}}
\put(46,25){\circle*{1.5}}
\put(86,25){\circle*{1.5}}
\put(6,25){\line(0,-1){10}}
\put(26,25){\line(0,-1){10}}
\put(46,25){\line(0,-1){10}}
\put(86,25){\line(0,-1){10}}
\put(6,15){\line(1,-2){5}}
\put(26,15){\line(1,-2){5}}
\put(46,15){\line(1,-2){5}}
\put(86,15){\line(1,-2){5}}
\put(6,15){\line(-1,-2){5}}
\put(26,15){\line(-1,-2){5}}
\put(46,15){\line(-1,-2){5}}
\put(86,15){\line(-1,-2){5}}
\put(26,25){\line(-1,0){20}}
\qbezier(46,25)(26,35)(6,25)
\qbezier(6,25)(46,45)(86,25)
\put(66,15){\makebox(0,0)[cc]{$\ldots$}}
\end{picture}
\hspace{5mm}
\linethickness{0.4pt}
\ifx\plotpoint\undefined\newsavebox{\plotpoint}\fi 
\begin{picture}(91.75,40)(0,0)
\put(21,5){\circle*{1.5}}
\put(41,5){\circle*{1.5}}
\put(81,5){\circle*{1.5}}
\put(31,5){\circle*{1.5}}
\put(51,5){\circle*{1.5}}
\put(91,5){\circle*{1.5}}
\put(26,5){\makebox(0,0)[cc]{$\ldots$}}
\put(26,0){\makebox(0,0)[cc]{$\underbrace{\hspace{1.0cm}}_{{\ell}_1}$}}
\put(46,5){\makebox(0,0)[cc]{$\ldots$}}
\put(46,0){\makebox(0,0)[cc]{$\underbrace{\hspace{1.0cm}}_{{\ell}_2}$}}
\put(86,5){\makebox(0,0)[cc]{$\ldots$}}
\put(86,0){\makebox(0,0)[cc]{$\underbrace{\hspace{1.0cm}}_{{\ell}_k}$}}
\put(6,15){\circle*{1.5}}
\put(26,15){\circle*{1.5}}
\put(46,15){\circle*{1.5}}
\put(86,15){\circle*{1.5}}
\put(6,25){\circle*{1.5}}
\put(26,25){\circle*{1.5}}
\put(46,25){\circle*{1.5}}
\put(86,25){\circle*{1.5}}
\put(6,25){\line(0,-1){10}}
\put(26,25){\line(0,-1){10}}
\put(46,25){\line(0,-1){10}}
\put(86,25){\line(0,-1){10}}
\put(26,15){\line(1,-2){5}}
\put(46,15){\line(1,-2){5}}
\put(86,15){\line(1,-2){5}}
\put(26,15){\line(-1,-2){5}}
\put(46,15){\line(-1,-2){5}}
\put(86,15){\line(-1,-2){5}}
\put(26,25){\line(-1,0){20}}
\qbezier(46,25)(26,35)(6,25)
\qbezier(6,25)(46,45)(86,25)
\put(66,15){\makebox(0,0)[cc]{$\ldots$}}
\end{picture}
\end{center}
\caption{
For the tree $T_{\rm even}$ on the left, 
we have $k=\frac{\gamma_t}{2}$,
$1\leq \ell_1,\ldots,\ell_k$, and $(\ell_1+1)+\ldots+(\ell_k+1)=n-k$,
while 
for the tree $T_{\rm odd}$ on the right, 
we have $k=\frac{\gamma_t-1}{2}$,
$1\leq \ell_1,\ldots,\ell_k$, and $(\ell_1+1)+\ldots+(\ell_k+1)=n-k-2$.}\label{fig1}
\end{figure}
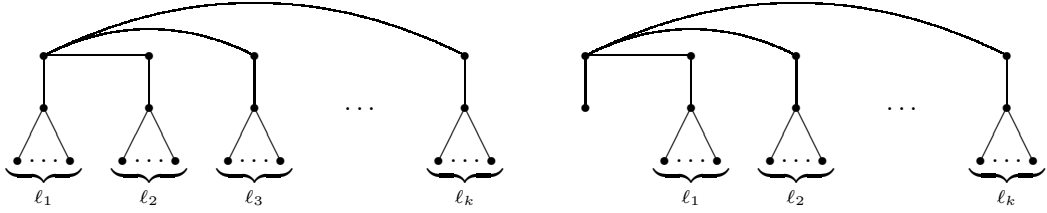
If $\gamma_t$ is even, say $\gamma_t=2k$,
then the tree $T_{\rm even}$ in the left of Figure \ref{fig1} satisfies
$$\sharp\gamma_t(T_{\rm even})=\prod_{i=1}^k(\ell_i+1)\leq 
\left(\frac{n-\frac{\gamma_t}{2}}{\frac{\gamma_t}{2}}\right)^{\frac{\gamma_t}{2}},$$
where we use that the geometric mean is at most the arithmetic mean.
Similarly,
if $\gamma_t$ is odd, say $\gamma_t=2k+1$,
then the tree $T_{\rm odd}$ in the right of Figure \ref{fig1} satisfies
$$\sharp\gamma_t(T_{\rm odd})=
\sum_{i=1}^k
\left(\prod_{j=1}^{i-1}\ell_j
\prod_{j=i+1}^k(\ell_j+1)\right)
\leq 
k\left(\frac{n-k-4}{k-1}\right)^{k-1}
=
\left(\frac{\gamma_t-1}{2}\right)
\left(\frac{n-\left(\frac{\gamma_t+7}{2}\right)}{\frac{\gamma_t-3}{2}}\right)^{\frac{\gamma_t-3}{2}}.$$
In view of these estimates, we pose the following.

\begin{conjecture}\label{conjecture1}
If a tree $T$ has order $n$ at least $2$ and 
total domination number $\gamma_t$,
then 
$$\sharp\gamma_t(T)\leq \left(\frac{n-\frac{\gamma_t}{2}}{\frac{\gamma_t}{2}}\right)^{\frac{\gamma_t}{2}}.$$
\end{conjecture}
As our first result, we show that Conjecture \ref{conjecture1}
holds up to a constant factor for bounded values of $\gamma_t$.
More precisely, we show the following.

\begin{theorem}\label{t:thm3}
If a forest $F$ has order $n$, no isolated vertex, and total domination number $\gamma_t$, then 
$$\sharp\gamma_t(F)\leq 
\left(8\sqrt{e}\, \right)^{\gamma_t}\left(\frac{n-\frac{\gamma_t}{2}}{\frac{\gamma_t}{2}}\right)^{\frac{\gamma_t}{2}}.$$
\end{theorem} 
The well known estimate $1+x\leq e^{x}$ implies 
$$\left(\frac{n-\frac{\gamma_t}{2}}{\frac{\gamma_t}{2}}\right)^{\frac{\gamma_t}{2}}
=\left(1+\frac{n-\gamma_t}{\frac{\gamma_t}{2}}\right)^{\frac{\gamma_t}{2}}
\leq e^{n-\gamma_t}.$$
In the following theorem we can show an upper bound that is a little better. But since $1+x<< e^{x}$ for large $x$, the estimate is not good for fixed $\gamma_t$ and large values of $n$. In this case Theorems \ref{t:thm3} and \ref{t:thm2} give better upper bounds. 


\begin{theorem}\label{t:thm1}
If a forest $F$ has order $n$, no isolated vertex, and total domination number $\gamma_t$, 
then 
$$\sharp\gamma_t(F) \leq (1+\sqrt{2})^{n-\gamma_t},$$ 
with equality if and only if every component of $F$ is $K_2$. 
\end{theorem}
Note that Theorem \ref{t:thm1} is only tight for $\gamma_t=n$,
which corresponds to the fact that $1+x=e^{x}$ only for $x=0$.
For $n$ divisible by $5$, 
the disjoint union of $\frac{n}{5}$ stars of order $5$
yields a forest $F$ with 
$\sharp\gamma_t(F)=4^\frac{n}{5}\approx 1.3195^n$.
Our third result comes close to that value.

\begin{theorem}\label{t:thm2}
If a forest $F$ has order $n$ and no isolated vertex, 
then $\sharp\gamma_t(F)\leq 1.4865^n$.
\end{theorem} 

Before we proceed to the proofs of our results,
we mention some related research.
Connolly et al.~\cite{cogagokake} gave bounds 
on the maximum number of minimum dominating sets for general graphs.
The maximum number of minimal dominating sets 
was studied by Fomin et al.~\cite{fogrpyst},
and the maximum number of general dominating sets 
by Wagner \cite{wa} and Skupie\'{n} \cite{sk},
and by Br\'{o}d and Skupie\'{n} \cite{brsk} for trees.
Krzywkowski and Wagner \cite{krwa} study 
the maximum number of total dominating sets for general graphs and trees.
For similar research concerning independent sets we refer to \cite{kogodo,wl,zi}.

The next section contains the proofs of our results.
We use standard graph theoretical terminology and notation.
An {\it endvertex} is a vertex of degree at most $1$,
and a {\it support vertex} is a vertex that is adjacent to an endvertex.

\section{Proofs}

For the proof of Theorem~\ref{t:thm3},
we need the following lemma. 

\begin{lemma}\label{lem1}
If $T$ is a tree of order $n$ at least $2$, and $B$ is a set of vertices of $T$ such that
\begin{enumerate}[(i)]
\item $|B\cap \{ u,v\}|\le 1$ for every $uv\in E(T)$, and
\item $|B\cap N_T(u)\}|\le 1$ for every $u\in V(T)$, 
\end{enumerate}
then $|B|\le \frac{n}{2}$.
\end{lemma}
\begin{proof} The proof is by induction on $n$.
If $T$ is a star, then (i) and (ii) imply $|B|\le 1\le \frac{n}{2}$.
Now, let $T$ be a tree that is not a star;
in particular, $n\ge 4$.
Let $uvw\ldots $ be a longest path in $T$.
By (i) and (ii), we have $|B\cap (N_T[v]\setminus \{ w\})|\le 1$.
By induction applied to the tree $T'=T-(N_T[v]\setminus \{ w\})$
and the set $B'=B\cap V(T')$,
we obtain
$|B|\le |B'|+\Big|B\cap (N_T[v]\setminus \{ w\})\Big|
\le \frac{n(T')}{2}+1
\le \frac{n}{2}.$
\end{proof}

We are now in a position to present 
the proof of Theorem~\ref{t:thm3}.

\begin{proof}[of Theorem~\ref{t:thm3}] 
Let $F$ be a forest of order $n$ and total domination number $\gamma_t$ such that $\sharp\gamma_t(F)$ is as large as possible.
Let $D$ be a $\gamma_t$-set of $F$.
Let $F'$ arise by removing from $F$ all endvertices of $F$
that do not belong to $D$.
For every $u\in D$, let
$L(u)=N_F(u)\setminus N_{F'}(u)$
and
$\ell(u)=|L(u)|$,
that is, $L(u)$ is the set of neighbors of $u$ in $D$
that are endvertices of $F$ that do not belong to $D$.
We call a vertex $u$ in $D$ {\it big} if $\ell(u)\ge 2$,
and we assume that
-- subject to the above conditions --
the forest $F$ and the set $D$ are chosen such that
the number $k$ of big vertices is as small as possible.

\setcounter{claim}{0}

\begin{claim}\label{c1forthm3}
No two big vertices are adjacent.
\end{claim}
\begin{proof}[of Claim \ref{c1forthm3}]
Suppose, for a contradiction, that $u$ and $v$ are adjacent big vertices.
Let $L'$ be a set of $\ell(u)-1$ vertices in $L(u)$, and let
$F'=F
-\{ ux:x\in L'\}+\{ vx:x\in L'\}$,
that is, we shift $\ell(u)-1$ neighbors of $u$ in $L(u)$ to $v$.
Clearly, the vertices $u$ and $v$ both belong
to every $\gamma_t$-set of $F$ and also
to every $\gamma_t$-set of $F'$.
This easily implies that a set of vertices of $F$
is a $\gamma_t$-set of $F$ if and only if
it is a $\gamma_t$-set of $F'$.
It follows that $D$ is a $\gamma_t$-set of $F'$
and that $\sharp\gamma_t(F)=\sharp\gamma_t(F')$.
Since $F'$ and $D$ lead to less than $k$ big vertices,
we obtain a contradiction to the choice of $F$ and $D$.
\end{proof}

\begin{claim}\label{c2forthm3}
No two big vertices have a common neighbor in $D$.
\end{claim}
\begin{proof}[of Claim \ref{c2forthm3}]
Suppose, for a contradiction, that $u$ and $w$ are big vertices
with a common neighbor $v$ in $D$.
Let
\begin{itemize}
\item $\sharp_u$ be the number of $\gamma_t$-sets of $F$
that contain a vertex from $L(u)$,
\item $\sharp_w$ be the number of $\gamma_t$-sets of $F$
that contain a vertex from $L(w)$, and
\item $\sharp_{\bar{u},\bar{w}}$
be the number of $\gamma_t$-sets of $F$
that contain no vertex from $L(u)\cup L(w)$.
\end{itemize}
In view of $v$,
no $\gamma_t$-set of $F$
contains a vertex from both sets $L(u)$ and $L(w)$,
which implies
$$\sharp\gamma_t(F)=\sharp_u+\sharp_w+\sharp_{\bar{u},\bar{w}}.$$
Note that $\frac{\sharp_u}{\ell(u)}$
is the number of subsets of $V(F)\setminus L(u)$ 
that can be extended to a $\gamma_t$-set of $F$
by adding one vertex from $L(u)$.
By symmetry, we may assume that
$\frac{\sharp_u}{\ell(u)}\le \frac{\sharp_w}{\ell(w)}$.
Again,
let $L'$ be a set of $\ell(u)-1$ vertices in $L(u)$, and let
$F'=F
-\{ ux:x\in L'\}+\{ wx:x\in L'\}$.
Similarly as before,
the vertices $u$ and $w$ both belong
to every $\gamma_t$-set of $F$ and also
to every $\gamma_t$-set of $F'$.
It follows that
$D$ is a $\gamma_t$-set of $F'$, and
that
$$\sharp\gamma_t(F')=\frac{\sharp_u}{\ell(u)}+\frac{\sharp_w}{\ell(w)}(\ell(u)+\ell(w)-1)+\sharp_{\bar{u},\bar{w}}
\ge \sharp_u+\sharp_w+\sharp_{\bar{u},\bar{w}}
=\sharp\gamma_t(F).$$
Since $F'$ and $D$ lead to less than $k$ big vertices,
this contradicts the choice of $F$ and $D$.
\end{proof}

\begin{claim}\label{c3forthm3}
$k\le \frac{\gamma_t}{2}$.
\end{claim}
\begin{proof}[of Claim \ref{c3forthm3}]
This follows immediately by applying Lemma \ref{lem1}
to each component of $F[D]$,
choosing $B$ as the set of big vertices in that component.
\end{proof}

Let $n'=n(F')$,
let $V_1'$ be the set of endvertices of $F'$,
let $n_1'=|V_1'|$, and
let $m$ be the number of edges of $F'$
between $D$ and $V(F')\setminus D$.
Since the vertices in $V_1'$
are either endvertices of $F$ that belong to $D$
or are adjacent to an endvertex of $F$,
we obtain that $V_1'\subseteq D$.
Since $D$ is a total dominating set, we obtain
\begin{eqnarray}\label{e1}
n'-\gamma_t=|V(F')\setminus D|\le m
\le \sum\limits_{u\in D}(d_{F'}(u)-1).
\end{eqnarray}
Since $F'$ is a forest with, say, $\kappa$ components,
\begin{eqnarray*}
n_1'&=&2\kappa+\sum\limits_{u\in V(F'):d_{F'}(u)\ge 2}(d_{F'}(u)-2)\\
&\ge& \sum\limits_{u\in D:d_{F'}(u)\ge 2}(d_{F'}(u)-2)\\
&=&\sum\limits_{u\in D:d_{F'}(u)\ge 2}d_{F'}(u)
-2(\gamma_t-n_1'),
\end{eqnarray*}
which implies
\begin{eqnarray}
2\gamma_t-n_1' &\ge &\sum\limits_{u\in D:d_{F'}(u)\ge 2}d_{F'}(u).\label{e2}
\end{eqnarray}
Now, we obtain
\begin{eqnarray}\label{e3}
n' & \stackrel{(\ref{e1})}{\leq} & \sum\limits_{u\in D}d_{F'}(u)
= \sum\limits_{u\in D:d_{F'}(u)\ge 2}d_{F'}(u)+n_1'
\stackrel{(\ref{e2})}{\leq} 2\gamma_t.
\end{eqnarray}
Let $u_1,\ldots,u_k$ be the big vertices.
By (\ref{e3}), the forest
$F''=F-\bigcup_{i=1}^kL(u_i)$
has order at most $3\gamma_t$.
Let $D''$ be a set of vertices of $F''$ that is a subset of some $\gamma_t$-set $D$ of $F$.
For every $i\in \{ 1,\ldots,k\}$,
if $u_i$ has a neighbor in $D''$,
then $D$ contains no vertex from $L(u_i)$,
otherwise,
the set $D$ contains exactly one vertex from $L(u_i)$.
This implies that each of the $2^{n(F'')}$ subsets of $V(F'')$
can be extended to a $\gamma_t$-set of $F$
in at most $\prod\limits_{i=1}^k\ell(u_i)$ many ways.

Since
\begin{enumerate}[{\it (i)}]
\item $n(F'')\le 3\gamma_t$,
\item the geometric mean is less or equal the arithmetic mean,
\item $\sum\limits_{i=1}^k\ell(u_i)=n-n(F'')\le n-\gamma_t\le n-\frac{\gamma_t}{2}$,
\item $\left(1+\frac{\frac{\gamma_t}{2}-k}{k}\right)^k\le e^{\frac{\gamma_t}{2}-k}\le e^{\frac{\gamma_t}{2}}$, and
\item $\frac{\frac{\gamma_t}{2}}{n-\frac{\gamma_t}{2}}\le 1$,\end{enumerate}
we obtain
\begin{eqnarray*}
\sharp\gamma_t(F)
& \le & 2^{n(F'')}\prod\limits_{i=1}^k\ell(u_i)\\
& \stackrel{(i)}{\leq} & 2^{3\gamma_t}\prod\limits_{i=1}^k\ell(u_i)\\
& \stackrel{(ii)}{\leq} & 2^{3\gamma_t}\left(\frac{1}{k}\sum\limits_{i=1}^k\ell(u_i)\right)^k\\
& \stackrel{(iii)}{\leq} & 2^{3\gamma_t}\left(\frac{n-\frac{\gamma_t}{2} }{k}\right)^k\\
& = &
2^{3\gamma_t}
\left(1+\frac{\frac{\gamma_t}{2}-k}{k}\right)^{k}
\left(\frac{\frac{\gamma_t}{2}}{n-\frac{\gamma_t}{2}}\right)^{\frac{\gamma_t}{2}-k}
\left(\frac{n-\frac{\gamma_t}{2}}{\frac{\gamma_t}{2}}\right)^{\frac{\gamma_t}{2}}\\
& \stackrel{(iv)}{\leq} &
2^{3\gamma_t}
e^{\frac{\gamma_t}{2}}
\left(\frac{\frac{\gamma_t}{2}}{n-\frac{\gamma_t}{2}}\right)^{\frac{\gamma_t}{2}-k}
\left(\frac{n-\frac{\gamma_t}{2}}{\frac{\gamma_t}{2}}\right)^{\frac{\gamma_t}{2}}\\
& \stackrel{\text{Claim }\ref{c:claim3},~(v)}{\leq} &
2^{3\gamma_t}
e^{\frac{\gamma_t}{2}}
\left(\frac{n-\frac{\gamma_t}{2}}{\frac{\gamma_t}{2}}\right)^{\frac{\gamma_t}{2}}\\
&=& \left(8\sqrt{e}\right)^{\gamma_t}\left(\frac{n-\frac{\gamma_t}{2}}{\frac{\gamma_t}{2}}\right)^{\frac{\gamma_t}{2}},
\end{eqnarray*}
which completes the proof.
\end{proof}

There is clearly some room for lowering $8\sqrt{e}$ to a smaller constant.
Since the dependence on $\gamma_t$ would still be exponential, 
we did not exploit this for the sake of simplicity.
It would be interesting to see whether the bound can be improved to 
$$\left(1+o\left(\frac{n}{\gamma_t}\right)\right)\left(\frac{n-\frac{\gamma_t}{2}}{\frac{\gamma_t}{2}}\right)^{\frac{\gamma_t}{2}}.$$
Note that Theorem \ref{t:thm3} implies
$$\sharp\gamma_t(T)\leq \left(\frac{n-\frac{\gamma_t}{2}}{\frac{\gamma_t}{2}}\right)^{\frac{\gamma_t}{2}+o\left(\frac{n}{\gamma_t}\right)}.$$
We proceed to our next proof.

\begin{proof}[of Theorem~\ref{t:thm1}] 
We proceed by induction on $n$. 
If $n=2$, then $F=K_2$, $\gamma_t = 2$, and $\sharp\gamma_t(F) = 1 = (1+\sqrt{2})^0 = (1+\sqrt{2})^{n-\gamma_t}$. 
Now, let $n \ge 3$.

\setcounter{claim}{0}
\begin{claim}\label{c:clm1}
If $F$ contains a component $T$ that is a star, 
then $\sharp\gamma_t(F) \le (1+\sqrt{2})^{n-\gamma_t}$, 
with strict inequality if $T$ has order at least~$3$. 
\end{claim}
\begin{proof}[of Claim \ref{c:clm1}]
Suppose that $F$ contains a component $T$ that is a star. Thus, $T = K_{1,t}$ for some $t \ge 1$. 
The forest $F' = F - V(T)$ 
has order $n'=n-t-1$,
no isolated vertex, and
total domination number $\gamma_t' = \gamma_t-2$. 
By induction, we obtain
\begin{eqnarray*}
\sharp\gamma_t(F) &=& t\cdot \sharp\gamma_t(F') 
\le t (1+\sqrt{2})^{n'-\gamma_t'}
=t (1+\sqrt{2})^{n-t-1-(\gamma_t-2)}\\
&=&(1+\sqrt{2})^{n-\gamma_t} (t (1+\sqrt{2})^{1-t}) \le (1+\sqrt{2})^{n-\gamma_t},
\end{eqnarray*}
where we use $t (1+\sqrt{2})^{1-t} \leq 1$ for $t = 1$ and $t\geq2$. 
Furthermore, if $t \ge 2$, then $t (1+\sqrt{2})^{1-t} < 1$, in which case $\sharp\gamma_t(F) < (1+\sqrt{2})^{n-\gamma_t}$.
\end{proof}

\begin{claim}\label{c:clm2}
If $F$ contains a component $T$ of diameter $3$,
then $\sharp\gamma_t(F) < (1+\sqrt{2})^{n-\gamma_t}$.
\end{claim}
\begin{proof}[of Claim \ref{c:clm2}]
Suppose that $F$ contains a component $T$ of diameter $3$.  
Note that $T$ has a unique minimum total dominating set. 
The forest $F' = F - V(T)$ 
has order $n'\leq n-4$,
no isolated vertex, and
total domination number $\gamma_t' = \gamma_t-2$. 
By induction, we obtain
$$
\sharp\gamma_t(F)=\sharp\gamma_t(F')\le (1+\sqrt{2})^{n'-\gamma_t'}\le (1+\sqrt{2})^{n-\gamma_t-2}<(1+\sqrt{2})^{n-\gamma_t}.$$
\end{proof}

By Claim~\ref{c:clm1} and Claim~\ref{c:clm2}, we may assume that there is a component of $F$ that has diameter at least~$4$, for otherwise the desired result follows. Let $T$ be such a component of $F$. Let $uvwxy\ldots r$ be a longest path in $T$,
and consider $T$ as rooted in $r$.
For a vertex $z$ of $T$, 
let $V_z$ be the set that contains $z$ and all its descendants.

\begin{claim}\label{c:clm3}
If $d_F(w) \ge 3$, then $\sharp\gamma_t(F) < (1+\sqrt{2})^{n-\gamma_t}$.
\end{claim}
\begin{proof}[of Claim \ref{c:clm3}]
Suppose that $d_F(w) \ge 3$, which implies that $w$ belongs to every $\gamma_t$-set of $F$, because either $w$ is a support vertex or $w$ is the only neighbor of two support vertices, that is no leaf. Let $v'$ be a child of $w$ distinct from $v$. 
Let $F' = F - V_{v'}$.
If $v'$ is an endvertex, then $F'$ 
has order $n'=n-1$,
no isolated vertex, and
total domination number $\gamma_t' = \gamma_t$. 
By induction, we obtain
$$\sharp\gamma_t(F) \le \sharp\gamma_t(F') \le (1+\sqrt{2})^{n'-\gamma_t'} = (1+\sqrt{2})^{n-\gamma_t-1} < (1+\sqrt{2})^{n-\gamma_t}.$$
If $v'$ is not an endvertex, then $F'$ 
has order $n'\leq n-2$,
no isolated vertex, and
total domination number $\gamma_t' = \gamma_t-1$. Note that if $T$ is a minimum total dominating set of $F$, $T-\{v\}$ is a total dominating set of $F'$, since $v'$ is a support vertex and $v$ and $w$ are part of every minimum total dominating set of $F$. 
By induction, we obtain
$$\sharp\gamma_t(F) \le \sharp\gamma_t(F') \le (1+\sqrt{2})^{n'-\gamma_t'} \le (1+\sqrt{2})^{n-\gamma_t-1} < (1+\sqrt{2})^{n-\gamma_t}.$$
In both cases, $\sharp\gamma_t(F) < (1+\sqrt{2})^{n-\gamma_t}$.
\end{proof}

By Claim~\ref{c:clm3}, we may assume that $d_F(w)=2$, 
for otherwise the desired result holds.

\begin{claim}\label{c:clm4}
If $d_F(v) \ge 3$, then $\sharp\gamma_t(F) < (1+\sqrt{2})^{n-\gamma_t}$.
\end{claim}
\begin{proof}[of Claim \ref{c:clm4}]
Suppose that $\ell=d_F(v)-1\ge 2$.
Let 
$F'=F - V_w$, 
$F''=F - (N_F(v)\setminus \{ w\})$, and 
$F'''=F- (V_w\cup \{ x\})$. 
See Figure \ref{fig2.1} for an illustration.

\begin{figure}[H]
\begin{minipage}[t]{0.25\textwidth}
\begin{center}
\begin{tikzpicture}
\path[use as bounding box] (-1,0.5) rectangle (2,7);

\node[fill, circle, inner sep=1.3pt, label=right:$v$] (v1) at (0.5,3) {};
\node[fill, circle, inner sep=1.3pt] (v2) at (0,2) {};
\node[fill, circle, inner sep=1.3pt] (v21) at (0.3,2) {}; 
\node[fill, circle, inner sep=1.3pt] (v3) at (1,2) {};
\node[fill, circle, inner sep=1.3pt, label=right:$w$] (v4) at (0.5,4) {};
\node[fill, circle, inner sep=1.3pt, label=right:$x$] (v5) at (0.5,5) {};
\node[fill, circle, inner sep=1.3pt, label=right:$y$] (v6) at (0.5,6) {};

\node[fill, circle, inner sep=1.3pt](h1) at (0,4) {};
\node[fill, circle, inner sep=1.3pt](h2) at (-0.4,4) {};
\node[fill, circle, inner sep=1.3pt](h3) at (0,5) {};
\node[fill, circle, inner sep=1.3pt](h4) at (-0.4,5) {};
\node(h5) at (0.5,6.5) {};  
  
\draw (v1) -- (v2);
\draw (v1) -- (v21);
\draw (v1) -- (v3);
\draw (v1) -- (v4);
\draw (v4) -- (v5);
\draw (v5) -- (v6);

\draw[loosely dotted] (v21)--(v3);

\draw[dashed] (v5)--(h1);
\draw[dashed] (v5)--(h2);
\draw[dashed] (v6)--(h3);
\draw[dashed] (v6)--(h4);
\draw[dashed] (v6)--(h5);

\draw decorate [decoration={name=brace,mirror}, yshift=2ex]  {(0,1.5) -- node[below=0.4ex] {$\ell$}  (1,1.5)};

\node[align=left] at (0.5,1) {The forest $F$};

\end{tikzpicture}
\end{center}
\end{minipage}\begin{minipage}[t]{0.25\textwidth}
\begin{center}
\begin{tikzpicture}
\path[use as bounding box] (-1,0.5) rectangle (2,7);
\node[fill, circle, inner sep=1.3pt, label=right:$x$] (v5) at (0.5,5) {};
\node[fill, circle, inner sep=1.3pt, label=right:$y$] (v6) at (0.5,6) {};

\node[fill, circle, inner sep=1.3pt](h1) at (0,4) {};
\node[fill, circle, inner sep=1.3pt](h2) at (-0.4,4) {};
\node[fill, circle, inner sep=1.3pt](h3) at (0,5) {};
\node[fill, circle, inner sep=1.3pt](h4) at (-0.4,5) {};
\node(h5) at (0.5,6.5) {};  
  
\draw (v5) -- (v6);

\draw[dashed] (v5)--(h1);
\draw[dashed] (v5)--(h2);
\draw[dashed] (v6)--(h3);
\draw[dashed] (v6)--(h4);
\draw[dashed] (v6)--(h5);

\node[align=left] at (0.5,1) {The forest $F'$};

\end{tikzpicture}
\end{center}

\end{minipage}\begin{minipage}[t]{0.25\textwidth}
\begin{center}
\begin{tikzpicture}[auto]

\path[use as bounding box] (-1,0.5) rectangle (2,7);
\node[fill, circle, inner sep=1.3pt, label=right:$v$] (v1) at (0.5,3) {};

\node[fill, circle, inner sep=1.3pt, label=right:$w$] (v4) at (0.5,4) {};
\node[fill, circle, inner sep=1.3pt, label=right:$x$] (v5) at (0.5,5) {};
\node[fill, circle, inner sep=1.3pt, label=right:$y$] (v6) at (0.5,6) {};

\node[fill, circle, inner sep=1.3pt](h1) at (0,4) {};
\node[fill, circle, inner sep=1.3pt](h2) at (-0.4,4) {};
\node[fill, circle, inner sep=1.3pt](h3) at (0,5) {};
\node[fill, circle, inner sep=1.3pt](h4) at (-0.4,5) {};
\node(h5) at (0.5,6.5) {};  
 
\draw (v1) -- (v4);
\draw (v4) -- (v5);
\draw (v5) -- (v6);

\draw[dashed] (v5)--(h1);
\draw[dashed] (v5)--(h2);
\draw[dashed] (v6)--(h3);
\draw[dashed] (v6)--(h4);
\draw[dashed] (v6)--(h5);

\node[align=left] at (0.5,1) {The forest $F''$};

\end{tikzpicture}
\end{center}

\end{minipage}\begin{minipage}[t]{0.25\textwidth}
\begin{center}
\begin{tikzpicture}[auto]
\path[use as bounding box] (-1,0.5) rectangle (2,7);
\node (v5) at (0.5,5) {};
\node[fill, circle, inner sep=1.3pt, label=right:$y$] (v6) at (0.5,6) {};

\node[fill, circle, inner sep=1.3pt](h1) at (0,4) {};
\node[fill, circle, inner sep=1.3pt](h2) at (-0.4,4) {};
\node[fill, circle, inner sep=1.3pt](h3) at (0,5) {};
\node[fill, circle, inner sep=1.3pt](h4) at (-0.4,5) {};
\node(h5) at (0.5,6.5) {};  
  
\draw[dashed] (v6)--(h3);
\draw[dashed] (v6)--(h4);
\draw[dashed] (v6)--(h5);

\node[align=left] at (0.5,1) {The forest $F'''$};

\end{tikzpicture}
\end{center}

\end{minipage}

\caption{The important details of the forests $F$, $F'$, $F''$ and $F'''$.}\label{fig2.1}
\end{figure}
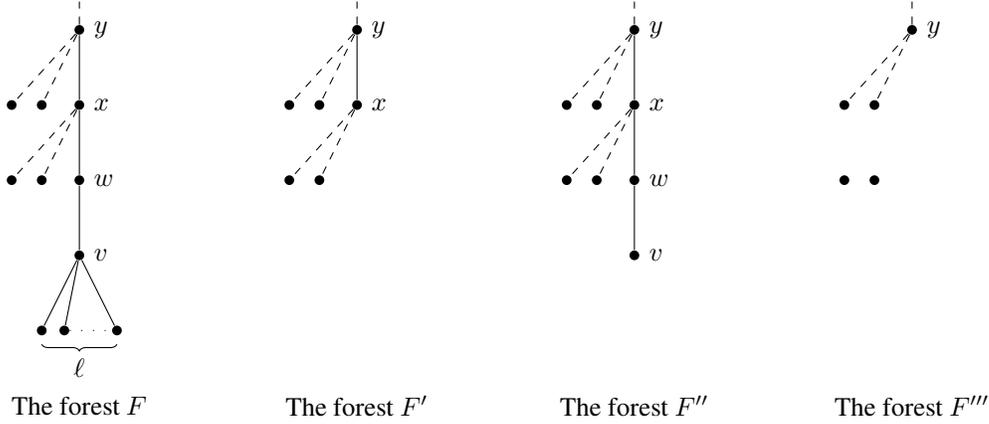

\begin{itemize}
\item There are at most $\ell\cdot \sharp\gamma_t(F')$ 
many $\gamma_t$-sets of $F$ that contain $v$ and a child of $v$
but do not contain $w$.
Furthermore, if such a $\gamma_t$-set exists, then
$F'$ has order $n'=n-\ell-2$,
no isolated vertex, and 
total domination number $\gamma_t'=\gamma_t-2$.
\item There are at most $\sharp\gamma_t(F'')$ 
many $\gamma_t$-sets of $F$ that contain $v$, $w$, and $x$. 
Furthermore, if such a $\gamma_t$-set exists, then
$F''$ has order $n''=n-\ell$,
no isolated vertex, and 
total domination number $\gamma_t''=\gamma_t-1$.
\item There are at most $\sharp\gamma_t(F''')$ 
many $\gamma_t$-sets of $F$ that contain both $v$ and $w$ 
but do not contain $x$. 
Furthermore, if such a $\gamma_t$-set exists, then
$F'''$ has order $n'''=n-\ell-3$,
no isolated vertex, and 
total domination number $\gamma_t'''=\gamma_t-2$.
\end{itemize}
Since all $\gamma_t$-sets of $F$ are of one of the three considered types, we obtain, by induction,
\begin{eqnarray*}
\sharp\gamma_t(F) & \le & \ell \cdot \sharp\gamma_t(F')+\sharp\gamma_t(F'')+\sharp\gamma_t(F''')\\
& \le & \ell (1+\sqrt{2})^{n-\ell-2-(\gamma_t-2)} + (1+\sqrt{2})^{n-\ell-(\gamma_t-1)} +(1+\sqrt{2})^{n-\ell-3-(\gamma_t-2)}\\
& = & (1+\sqrt{2})^{n-\gamma_t}(1+\sqrt{2})^{-\ell-1}(\ell (1+\sqrt{2})+(1+\sqrt{2})^2+1)\\
& < & (1+\sqrt{2})^{n-\gamma_t},
\end{eqnarray*}
where we use $\ell (1+\sqrt{2})+(1+\sqrt{2})^2+1 < (1+\sqrt{2})^{\ell+1}$ for all $\ell\ge 2$.
\end{proof}

By Claim~\ref{c:clm4}, 
we may assume that $d_F(v)=2$, 
for otherwise the desired result holds.

\begin{claim}\label{c:clm5}
If $x$ is a support vertex, then $\sharp\gamma_t(F) < (1+\sqrt{2})^{n-\gamma_t}$.
\end{claim}
\begin{proof}[of Claim \ref{c:clm5}]
Suppose that $x$ is a support vertex,
which implies that $v$ and $x$ belong to every $\gamma_t$-set of $F$.
Let 
$F'=F- V_w$ and 
$F''=F - (N_F[v]\cup N_F[x])$. 
\begin{itemize}
\item There are at most $\sharp\gamma_t(F')$ 
many $\gamma_t$-sets of $F$ that contain $u$
but do not contain $w$.
Furthermore, if such a $\gamma_t$-set exists, then
$F'$ has order $n'=n-3$,
no isolated vertex, and 
total domination number $\gamma_t'=\gamma_t-2$.
\item There are at most $\sharp\gamma_t(F')$
many $\gamma_t$-sets of $F$ that contain $w$ 
and at least one other neighbour of $x$. 
Furthermore, if such a $\gamma_t$-set exists, then
$F'$ has order $n'=n-3$,
no isolated vertex, and 
total domination number $\gamma_t'=\gamma_t-2$.
\item There are at most $\sharp\gamma_t(F'')$ 
many $\gamma_t$-sets of $F$ that contain $w$ 
and no other neighbour of $x$. 
Furthermore, if such a $\gamma_t$-set exists, then
$F''$ has order $n''\leq n-5$,
no isolated vertex, and 
total domination number $\gamma_t''=\gamma_t-3$.
\end{itemize}
Since all $\gamma_t$-sets of $F$ are of one of the three considered types, we obtain, by induction,
\begin{eqnarray*}
\sharp\gamma_t(F) & \le & 2\sharp\gamma_t(F')+\sharp\gamma_t(F'')
< 2 (1+\sqrt{2})^{n-3-(\gamma_t-2)} + (1+\sqrt{2})^{n-5-(\gamma_t-3)}\\
&=&(1+\sqrt{2})^{n-\gamma_t}(1+\sqrt{2})^{-2}(2(1+\sqrt{2})+1)=(1+\sqrt{2})^{n-\gamma_t},
\end{eqnarray*}
where we use $2(1+\sqrt{2})+1=(1+\sqrt{2})^2$. Note that in $F'$ there is a component that contains a path of length two, in particular not every component of $F'$ is a $K_2$. 
\end{proof}

By Claim~\ref{c:clm5}, 
we may assume that $x$ is not a support vertex, 
for otherwise the desired result holds.

\begin{claim}\label{c:clm6}
If $x$ has a child that is a support vertex, then $\sharp\gamma_t(F) < (1+\sqrt{2})^{n-\gamma_t}$.
\end{claim}
\begin{proof}[of Claim \ref{c:clm6}]
Suppose that $x$ has a child $w'$ that is a support vertex. 
Clearly, the vertex $w'$ is distinct from $w$
and belongs to every $\gamma_t$-set of $F$. 
The forest $F'=F-V_w$
has order $n'=n-3$,
no isolated vertex, and 
total domination number $\gamma_t'=\gamma_t-2$.
By induction, we obtain
\begin{eqnarray*}
\sharp\gamma_t(F)&=&2\sharp\gamma_t(F')
\le 2 (1+\sqrt{2})^{n-3-(\gamma_t-2)}
=(1+\sqrt{2})^{n-\gamma_t}2(1+\sqrt{2})^{-1}
<(1+\sqrt{2})^{n-\gamma_t },
\end{eqnarray*}
where we use $2<(1+\sqrt{2})$.
\end{proof}

By Claim \ref{c:clm6}, 
we may assume that no child of $x$ is a support vertex, 
for otherwise the desired result holds.
Together with Claims~\ref{c:clm3} and~\ref{c:clm4},
we may assume that the subforest of $F$ induced by $V_x$
arises from a star $K_{1,q}$ for some $q\geq 1$
by subdividing every edge twice.
Let 
$F'=F - V_x$, 
$F''=F- (V_x\cup \{y\})$, and 
$F'''=F- (V_x\cup N_F[y])$. 
\begin{itemize}
\item There are at most $2^q \sharp\gamma_t(F')$ 
many $\gamma_t$-sets of $F$ that do not contain $x$. 
Furthermore, if such a $\gamma_t$-set exists, then
$F'$ has order $n'=n-3q-1$,
no isolated vertex, and 
total domination number $\gamma_t'=\gamma_t-2q$.
\item There are at most $(2^q-1)\sharp\gamma_t(F'')$ 
many $\gamma_t$-sets of $F$ that contain $x$ but do not contain $y$. 
Furthermore, if such a $\gamma_t$-set exists, then
$F''$ has order $n''=n-3q-2$,
no isolated vertex, and 
total domination number $\gamma_t''=\gamma_t-2q-1$.
\item There are at most $2^q\sharp\gamma_t(F''')$ 
many $\gamma_t$-sets of $F$ that contain both $x$ and $y$. 
Furthermore, if such a $\gamma_t$-set exists, then
$F'''$ has order $n'''\leq n-3q-3$,
no isolated vertex, and 
total domination number $\gamma_t'''=\gamma_t-2q-2$.
\end{itemize}
Since all $\gamma_t$-sets of $F$ are of one of the three considered types, we obtain, by induction,
\begin{eqnarray*}
\sharp\gamma_t(F) & \le & 2^q \sharp\gamma_t(F')+(2^q-1)\sharp\gamma_t(F'')+2^q \sharp\gamma_t(F''')\\
& \le & 2^q  (1+\sqrt{2})^{n-3q-1-(\gamma_t-2q)} \\
&& + (2^q-1) (1+\sqrt{2})^{n-3q-2-(\gamma_t-2q-1)} \\
&& +2^q (1+\sqrt{2})^{n-3q-3-(\gamma_t-2q-2)}  \\
& = & (1+\sqrt{2})^{n-\gamma_t}(1+\sqrt{2})^{-q-1} (2^q+2^q-1+2^q)\\
& < & (1+\sqrt{2})^{n-\gamma_t},
\end{eqnarray*}
where we use $3\cdot 2^q-1< (1+\sqrt{2})^{q+1}$ for all $q \ge 1$. 
This completes the proof of Theorem~\ref{t:thm1}.
\end{proof}

We proceed to the proof of Theorem~\ref{t:thm2},
which uses exactly the same approach as Theorem~\ref{t:thm1}.

\begin{proof}[of Theorem~\ref{t:thm2}] 
By induction on $n$,
we show that $\sharp\gamma_t(F) \le \beta^n$, 
where $\beta$ is the unique positive real solution of the equation 
$2\beta + \beta^3 +1 =\beta^5,$
that is, $\beta \approx 1.4865$. 
If $n=2$, then $F=K_2$ and $\sharp\gamma_t(F)=1<\beta^2$.
Now, let $n \ge 3$.
\setcounter{claim}{0}

\begin{claim}\label{c:claim1}
If $F$ contains a component $T$ that is a star, then $\sharp\gamma_t(F) \le \beta^n$. 
\end{claim}
\begin{proof}[of Claim \ref{c:claim1}]
Suppose that $F$ contains a component $T$ that is a star. 
Thus, $T = K_{1,t}$ for some $t \ge 1$. 
The forest $F' = F - V(T)$ 
has order $n'=n-t-1$
and no isolated vertex.
By induction, we obtain
$\sharp\gamma_t(F) = t\cdot \sharp\gamma_t(F') \le t \beta^{n-t-1}\le \beta^n,
$
where we use $t\leq \beta^{t+1}$.
\end{proof}

\begin{claim}\label{c:claim2}
If $F$ contains a component $T$ of diameter $3$, 
then $\sharp\gamma_t(F) \le \beta^n$.
\end{claim}
\begin{proof}[of Claim \ref{c:claim2}]
Suppose that $F$ contains a component $T$ of diameter $3$.
The forest $F' = F - V(T)$ 
has order $n'\le n-4$ and no isolated vertex. 
By induction, we obtain
$\sharp\gamma_t(F) = \sharp\gamma_t(F') \le  \beta^{n'}< \beta^n.$
\end{proof}

By Claim~\ref{c:clm1} and Claim~\ref{c:clm2}, we may assume that every component of $F$ has diameter at least~$4$, for otherwise the desired result follows. Let $T$ be an arbitrary component of $F$. Let $uvwxy\ldots r$ be a longest path in $T$,
and consider $T$ as rooted in $r$.
For a vertex $z$ of $T$, 
let $V_z$ be the set that contains $z$ and all its descendants.

\begin{claim}\label{c:claim3}
If $d_F(w) \ge 3$, then $\sharp\gamma_t(F) \le \beta^n$. 
\end{claim}
\begin{proof}[of Claim \ref{c:claim3}]
Suppose that $d_F(w) \ge 3$, 
which implies that $w$ belongs to every $\gamma_t$-set of $F$. Let $v'$ be a child of $w$ distinct from $v$.
The forest $F' = F - V_{v'}$ has order $n'<n$
and no isolated vertex.
Since $\sharp\gamma_t(F) \le \sharp\gamma_t(F')$,
we obtain, by induction, 
$
\sharp\gamma_t(F) \le \sharp\gamma_t(F') \le \beta^{n'}<\beta^{n}.
$
\end{proof}

By Claim~\ref{c:claim3}, we may assume that $d_F(w)=2$, 
for otherwise the desired result holds.

\begin{claim}\label{c:claim4}
If $d_F(v) \ge 3$, then $\sharp\gamma_t(F) \le \beta^n$.
\end{claim}
\begin{proof}[of Claim \ref{c:claim4}]
Suppose that $\ell=d_F(v)-1\ge 2$.
Arguing exactly as in the proof of Claim~\ref{c:clm4} 
in the proof of Theorem~\ref{t:thm1}
using the forests $F'$, $F''$, and $F'''$,
we obtain, by induction,
\begin{eqnarray*}
\sharp\gamma_t(F) & \le & \ell \cdot \sharp\gamma_t(F')+\sharp\gamma_t(F'')+\sharp\gamma_t(F''')\\
& \le & \ell   \beta^{n-\ell-2} + \beta^{n-\ell} +\beta^{n-\ell-3}\\
& = & \beta^{n}\beta^{-\ell-3}\left(\ell  \beta+\beta^3+1\right)\\
& \le & \beta^n,
\end{eqnarray*}
where we use $\ell\beta+\beta^3+1\le \beta^{\ell+3}$ for all $\ell\ge 2$;
in fact, this inequality is the reason for the specific choice of $\beta$.
\end{proof}

By Claim~\ref{c:claim4},
we may assume that $d_F(v)=2$, 
for otherwise the desired result holds.

\begin{claim}\label{c:claim5}
If $x$ is a support vertex, then $\sharp\gamma_t(F)  \le \beta^n$.
\end{claim}
\begin{proof}[of Claim \ref{c:claim5}]
Suppose that $x$ is a support vertex.
Arguing exactly as in the proof of Claim~\ref{c:clm5} 
in the proof of Theorem~\ref{t:thm1}
using the forests $F'$ and $F''$,
we obtain, by induction,
\begin{eqnarray*}
\sharp\gamma_t(F) & \le & 2 \sharp\gamma_t(F')+\sharp\gamma_t(F'')
\le 2 \beta^{n-3} + \beta^{n-5}
= \beta^{n}\beta^{-5}(2\beta^2+1)
\le \beta^{n},
\end{eqnarray*}
where we use $2\beta^2+1\le \beta^5$.
\end{proof}

By Claim~\ref{c:claim5}, 
we may assume that $x$ is not a support vertex, 
for otherwise the desired result holds.

\begin{claim}\label{c:claim6}
If $x$ has a child that is a support vertex, then $\sharp\gamma_t(F) \leq \beta^n$.
\end{claim}
\begin{proof}[of Claim \ref{c:claim6}]
Suppose that $x$ has a child $w'$ that is a support vertex.
Arguing exactly as in the proof of Claim~\ref{c:clm6} 
in the proof of Theorem~\ref{t:thm1}
using the forest $F'$, 
we obtain, by induction,
\begin{eqnarray*}
\sharp\gamma_t(F)&=&2\sharp\gamma_t(F')
\leq 2 \beta^{n-3}
= \beta^{n} 2\beta^{-3}
\leq \beta^{n},
\end{eqnarray*}
where we use $2<\beta^3$.
\end{proof}

Now, arguing exactly 
as at the end of the proof of Theorem~\ref{t:thm1}
using the forests $F'$, $F''$, and $F'''$,
we obtain, by induction,
\begin{eqnarray*}
\sharp\gamma_t(F) & \le & 2^q \sharp\gamma_t(F')+(2^q-1)\sharp\gamma_t(F'')+2^q \sharp\gamma_t(F''')\\
& \le & 2^q  \beta^{n-3q-1} + (2^q-1) \beta^{n-3q-2} +2^q \beta^{n-3q-3}\\
& = & \beta^{n}\beta^{-3q-3}\left(2^q\beta^2+(2^q-1)\beta+2^q\right)\\
& \le & \beta^{n}\beta^{-3q-3}2^q\left(\beta^2+\beta+1\right)\\
& \le & \beta^n,
\end{eqnarray*}
where we use 
$2^q \left(\beta^2+\beta+1\right)\le \beta^{3q+3}$ for all $q\ge 1$. 
\end{proof}


\begin{thebibliography}{99}

\bibitem{aldamora} J.D. Alvarado, S. Dantas,  E. Mohr, and D. Rautenbach, On the maximum number of minimum dominating sets in forests, Discrete Mathematics 342 (2019) 934--942.
\bibitem{bi} A. Bie\'{n}, Properties of gamma graphs of trees, presentation at the 17th Workshop on Graph Theory Colourings, Independence and Domination (CID 2017), Piechowice, Poland.
\bibitem{brsk} D. Br\'{o}d and Z. Skupie\'{n}, Trees with extremal numbers of dominating sets, The Australasian Journal of Combinatorics 35 (2006) 273--290.
\bibitem{cogagokake} S. Connolly, Z. Gabor, A. Godbole, B. Kay, and T. Kelly, Bounds on the maximum number of minimum dominating sets, Discrete Mathematics 339 (2016) 1537--1542.
\bibitem{fogrpyst} F.V. Fomin, F. Grandoni, A.V. Pyatkin, and A.A. Stepanov, Bounding the number of minimal dominating sets: a measure and conquer approach, Lecture Notes in Computer Science 3827 (2005) 573--582.
\bibitem{frhehehu} G.H. Fricke, S.M. Hedetniemi, S.T. Hedetniemi, and K.R. Hutson, $\gamma$-graphs of graphs, Discussiones Mathematicae Graph Theory 31 (2011) 517--531.
\bibitem{hahesl} T.W. Haynes, S.T. Hedetniemi, and P. Slater, Fundamentals of domination in graphs, Marcel Dekker, Inc., New York, 1998.
\bibitem{heye} M.A. Henning and A. Yeo, Total Domination in Graphs, Springer 2013.
\bibitem{krwa} M. Krzywkowski and S. Wagner, Graphs with few total dominating sets, Discrete Mathematics 341 (2018) 997--1009.
\bibitem{kogodo} K.M. Koh, C.Y Goh, and F.M. Dong, The maximum number of maximal independent sets in unicyclic connected graphs, Discrete Mathematics 308 (2008) 3761--3769.
\bibitem{sk} Z. Skupie\'{n}, Majorization and the minimum number of dominating sets, Discrete Applied Mathematics 165 (2014) 295--302.
\bibitem{wa} S. Wagner, A note on the number of dominating sets of a graph, Utilitas Mathematica 92 (2013) 25--31.
\bibitem{wl} I. W\l och, Trees with extremal numbers of maximal independent sets including the set of leaves, Discrete Mathematics 308 (2008) 4768--4772.
\bibitem{zi} J. Zito, The structure and maximum number of maximum independent sets in trees, Journal of Graph Theory 15 (1991) 207--221.
\end{thebibliography}
\end{document}